\documentclass[12pt, reqno]{amsart}
\usepackage{amsmath, amsthm, amscd, amsfonts, amssymb, graphicx, xcolor}
\usepackage{mathrsfs}
\usepackage{float}
\usepackage[bookmarksnumbered, colorlinks, plainpages]{hyperref}
\usepackage[margin=1.2in]{geometry}
\usepackage{subfig}

\usepackage{algorithmicx}
\newtheorem{theorem}{Theorem}[section]
\newtheorem{lemma}[theorem]{Lemma}

\newtheorem{corollary}[theorem]{Corollary}
\theoremstyle{definition}
\newtheorem{definition}[theorem]{Definition}

\theoremstyle{remark}

\numberwithin{equation}{section}
\usepackage{algorithm}
\usepackage{algpseudocode}
\begin{document}
\setcounter{page}{1}

\title[Visibility of lattice points]{Visibility of Lattice points across polynomials}

\author[Chahat Ahuja]{Chahat Ahuja$^{1}$}

\address{$^{1}$ Department of Mathematics, IIITD, New Delhi.}
\email{\textcolor[rgb]{0.00,0.00,0.84}{chahat22138@iiitd.ac.in}}

\begin{abstract}
The visibility of lattice points from the origin along a polynomial family of curves constitutes a significant generalization of visibility along straight lines. Following the classical notion (density $=\zeta(2)^{-1}$ \raisebox{0.2cm}{\cite{Sylvester}}) and its generalisation to monomial curves $y=ax^b$ (density $=\zeta(b+1)^{-1}$ \raisebox{0.2cm}{\cite{J.Christopher}}), We work with  the family
\[
\mathcal{F}(a_n,\dots,a_1)=\bigl\{\,y=q(a_n x^n+\cdots+a_1 x)\mid q\in\mathbb{Q}^+\bigr\},
\]
We propose a novel criterion, based on a polynomial--gcd condition, which establishes a lower bound for the number of $F$-visible points in $\mathbb{N}^2$. Conversely, we also derive results for the polynomial curve, specifying the conditions under which a given point becomes the next visible lattice point on the curve.Moreover, utilizing the principle of inclusion--exclusion, an exact double-sum formula is derived for 
$$\#\{(a,b)\leq N :(a,b) \hspace{0.1cm}\text{is}\hspace{0.1cm} \mathcal{F}(a_{n} , a_{n-1} \cdots , a_{1})\text{--visible}\}$$
The framework is further extended to address additional problems. The paper concludes by posing several open questions regarding gap distributions and quantitative bounds for non-visible points within this geometric context. This extension provides a more robust theoretical foundation for understanding lattice point visibility in non-linear settings, moving beyond the traditional focus on linear or monomial relationships.
\end{abstract}
 \maketitle

\section{Introduction and Background}

\noindent A point $(a, b)$ in the two-dimensional integer lattice $\mathbb{Z}^{2}$ is said to be visible from the origin if there exists no other point of $\mathbb{Z}^{2}$ on the line segment joining the origin and the point. The set of visible lattice points is in one-to-one correspondence with lines passing through the origin of rational slope (including infinite slope).\\
If,we restrict our attention only to points in $\mathbb{N}^{2}$,then
visible lattice points in $\mathbb{N}^{2}$ are in bijection with family of lines
\vspace{-0.7cm}
\begin{center}
$$\{y = qx \hspace{0.1cm} | \hspace{0.1cm} q \in \mathbb{Q}^{+}\}. \quad (1)$$
\end{center}
A classical result proved by Sylvester \cite{Sylvester} in 1883 indicates that the density of visible lattice points in $\mathbb{Z}^{2}$ is $\frac{1}{\zeta(2)} = \frac{6}{\pi^2}$ where $\zeta(s)$ is the Riemann zeta function.\\
In fact, a similar argument establishes that for $k \ge 2$  the proportion of visible lattice points in $\mathbb{Z}^{k}$ (analogously defined) is given by $\frac{1}{\zeta(k)}$ \cite{J.Christopher}.\\
\vspace{-0.021cm}
\noindent In $2017$, Goins, Harris, Kubik and Mbirika generalised the classical definition of lattice point visibility by fixing $b \in  \mathbb{N}$ and considering curves of the form $f(x) = ax^{b}$ with $a \in \mathbb{Q}$ \cite{Goins}. In this new setting, a lattice point $(r, s) \in \mathbb{N}^{2}$ is said to be $b$-visible if it lies on the graph of a curve of the form $f(x) = ax^{b}$ with $a \in \mathbb{Q}$ and no other integer lattice points are lying on this curve between $(0, 0)$ and $(r, s)$. Hence, setting $b = 1$ recovers the classical
definition of lattice point visibility. In the $b$-visibility setting, Goins et. al. established that the proportion of $b$-visible lattice points in $\mathbb{N}^2$
is given by $\frac{1}{\zeta(b+1)}$ .\\
\vspace{-0.021cm}
Harris and Omar\cite{P.E. Harris} expanded this work to power functions with rational exponents by establishing that the proportion of $\frac{b}{a}$-visible lattice points in $\mathbb{N}^2_{a}$ (the set non-negative integers that are  $a$th powers) is given
by $\frac{1}{\zeta(b+1 )}$, and the proportion of $\frac{-b}{a}$-visible lattice points in $\mathbb{N}^2_a$ is given by $\frac{1}{\zeta(b)}$ \\
\vspace{-0.018cm}
We aim to extend this notion to polynomial families by exploring different methods for the visibility of lattice points around various families of polynomials with non-negative integer coefficients.We use notations, key definitions, and concepts related to polynomial families and their visibility, as introduced in the work of Sneha Chaubey, Ashish Kumar Pandey, and Shvo Regavim \cite{Sneha} on the density of visible lattice points along polynomials.\\

As mentioned , the notion of visibility  for points in $\mathbb{N}^{2}$  is generalized to the family $\mathcal{F}$ in Definition 1.1 as follows.\\
\\
\textbf{Definition 1.1.} \emph{For a fixed vector} $(a_n, a_{n-1},\cdots , a_1) \in \mathbb{Z}^{n}$ with $a_n \neq 0$, $a_i \ge 0$ $\emph{for all}$ $1 \leq i \leq n$, $\emph{and}$ $\gcd(a_n, a_{n-1}, \cdots, a_1) = 1$, $\emph{let}$

$$\mathcal{F}(a_n, a_{n-1}, \cdots , a_1) := \{y = q(a_nx^{n} + a_{n-1}x^{n-1} + \cdots + a_1x) \hspace{0.1cm} | \hspace{0.1cm} q \in \mathbb{Q}^{+}
\}.$$\\
\textbf{Definition 1.2.} \emph{Consider the family} $\mathcal{F}(a_n, a_{n-1}, \cdots , a_1)$ \emph{as in Definition 1.1. A point} $(r, s)$ \emph{in} $\mathbb{N}^2$ \emph{is said to be} $\mathcal{F}(a_n, a_{n-1}, \cdots , a_1)$\emph{-visible if there exists a ,} $q \in \mathbb{Q}^{+}$ \emph{such that} $s = q(a_{n}r^{n} +a_{n-1}r^{n-1} +\cdots +a_1r)$ \emph{and there is no other point of} $\mathbb{N}^{2}$ \emph{on the curve} $$y = q(a_nx^{n} + a_{n-1}x^{n-1} + \cdots + a_1x)$$ \emph{between the origin and} $(r, s)$. \\ \emph{Denote the set of all} $\mathcal{F}(a_n, a_{n-1}, \cdots , a_1)$ \emph{-visible points in} $\mathbb{N}^{2}$ \emph{by}  $\mathcal{L}(a_n, a_{n-1}, \cdots , a_1).$\\
\\
\textbf{Definition 1.3.} \emph{The density of a set} $\mathcal{S} \subseteq \mathbb{N}^{2}$
\emph{is} $$\mathsf{dens}(S) = \lim_{ N \rightarrow \infty} \frac{|S  \cap R_{N} |}{N^2}$$
\emph{provided the limit exists, where} $N \in \mathbb{N}$, $\left| .\right|$ \emph{denotes the cardinality of a set, and} $R_{N} = \{(a, b) \in \mathbb{Z}^{2}
 \hspace{0.1cm} | \hspace{0.1cm}  1 \leq a, b \leq N\}$. \\

\section{Point to polynomial Visibility}
In this section, we try to explore the converse of the classical approach of studying the visibility of lattice points, i.e., rather than finding visible lattice points for a respective polynomial, we work on finding a polynomial for which a given $(a,b) \in \mathbb{N}^{2}$ is visible.\\

\begin{theorem}
Let $(a,b) \in \mathbb{N}^{2}$, and let $\ell$ be a prime number satisfying $\ell > \max(a,b)$. Suppose that the base representation $a$ of $\ell$ is given by
\[
\ell = a_{d}a_{d-1}\cdots a_{1},
\]
where $a_{i}$ are the digits in base-$a$ expansion of $\ell$, i.e.,
\[
0 \leq a_{i} \leq a-1 \quad \text{for all } i.
\]
define the polynomial
\[
\mathcal{P}(x) := \frac{xb}{\ell a}\left(a_{d}x^{d-1} + a_{d-1}x^{d-2} + \cdots + a_{1}\right).
\]
Then the point $(a,b)$ lies on the curve $\mathcal{P}(x)$, and there is no lattice point on the segment of $\mathcal{P}(x)$ that joins $(0,0)$ to $(a,b)$.

\end{theorem}

\vspace{0.35cm}
\begin{proof}
We have
\[
\ell = a_{d}a^{d-1} + a_{d-1}a^{d-2} + \cdots + a_{1},
\]
which immediately gives
\[
\mathcal{P}(a) = b.
\]
Furthermore, for any $0 < t < a$, we have
\[
\frac{tb}{\ell a} \cdot (a_{d}t^{d-1} + a_{d-1}t^{d-2} + \cdots + a_{1}) \notin \mathbb{Z},
\]
because
\[
\gcd\left(\ell,\, a_{d}t^{d-1} + a_{d-1}t^{d-2} + \cdots + a_{1}\right) = 1,
\]
as $\ell$ is a prime and
\[
a_{d}t^{d-1} + a_{d-1}t^{d-2} + \cdots + a_{1} < \ell.
\]
Moreover,
\[
\gcd\left(\ell,\, a_{d}t^{d} + a_{d-1}t^{d-1} + \cdots + a_{1}t\right) = 1,
\]
since $\ell$ is a prime larger than $\max(a,b)$, implying $\ell > t$, and also $\gcd(\ell, b) = 1$.
Hence, the claim follows. \qed
\end{proof}

The polynomial
\[
a_{d}x^{d-1} + a_{d-1}x^{d-2} + \cdots + a_{1}
\]
is irreducible in $\mathbb{Z}[x]$, as proved by Brillhart, Filaseta, and Odlyzko~\cite{Filaseta}. 
\vspace{0.4cm} \\
\textbf{Illustration} of $\textbf{Theorem 2.1.}$\\

 1. Let $(a,b) = (3,5)$ and choose the prime $\ell = 7 > \max\{3,5\}$.  Write $7$ in base $a=3$:\\

  $7 = 2\cdot 3^1 + 1 \cdot 3^0,
  \quad (a_2,a_1) = (2,1), \; d = 2.$
Then
\[
  \mathcal P(x)
  = \frac{b}{\ell\,a}\,x\,(a_2 x^{d-1} + a_1)
  = \frac{5x}{7\cdot3}\,(2x+1)
  = \frac{5x(2x+1)}{21}.
\]
For $x=1,2$ we compute
\[
  \mathcal P(1) = \frac{5\cdot1\cdot3}{21} = \frac{5}{7}, 
  \quad
  \mathcal P(2) = \frac{5\cdot2\cdot5}{21} = \frac{50}{21},
\]
neither of which is an integer.  At $x=3$, $\mathcal P(3)=5$, recovering the endpoint $(3,5)$.  Hence, no interior lattice point lies on the curve.\\

\noindent
\textbf{Remark.} Observe that the statement of \textbf{Theorem~2.1} holds for every prime $\ell$ satisfying $\ell > \max(a,b)$. Consequently, there exist infintely many primes $\ell$ for which \textbf{Theorem~2.1} remains valid.
\\

This theorem shows that for any given lattice point \((a,b) \in \mathbb{N}^2\), one can always construct a polynomial \(\mathcal{P}(x)\) such that the point \((a,b)\) lies on its graph and the segment from \((0,0)\) to \((a,b)\) contains no other lattice points that is, \((a,b)\) is \emph{visible} on this curve.\\\\
Although this result formally demonstrates that visibility can always be achieved by tailoring a polynomial to a point, it is somewhat artificial in nature. The construction depends on the choice of a large prime \(\ell\) and encodes information about \(a\) in its base-\(a\) representation modulo \(\ell\), making it highly specific to the chosen point. Thus, although visibility is ensured, it does not address the more natural and mathematically rich question of determining which lattice points are visible \emph{for a family of polynomials}. That problem of studying visibility patterns and densities for general or structurally constrained polynomial families is far more interesting and subtle, revealing deeper arithmetic and geometric properties of the curve. We, now, provide a structured proof for the existence of such a relationship between polynomial curves and points. \\

\begin{lemma}

Let $(a_1,a_2,\dots,a_n)\in\mathbb{N}^n$ be any fixed lattice point.Choose a prime
\[
  \ell \;>\;\max\{a_1,a_2,\dots,a_n\},
\]
and its base‑$a_1$ expansion
\[
  \ell = a_d\,a_{d-1}\,\dots\,a_1
  \quad(\text{digits }0\le a_i < a_1).
\]
For each coordinate $2\le k\le n$,
define the polynomial
\[
  \mathcal P_k(x)
  \;=\;
  \frac{a_k}{\ell\,a_1}\;x\;\bigl(a_d\,x^{d-1}+a_{d-1}\,x^{d-2}+\cdots+a_1\bigr)
\]
Consider the curve
\[
  \Gamma: x \;\longmapsto\;
  \bigl(x,\;\mathcal P_2(x),\;\mathcal P_3(x),\;\dots,\;\mathcal P_n(x)\bigr),
  \quad x\in[0,a_1].
\]
Then:

(a) $\Gamma(0) = (0,0,\dots,0)$ and $\Gamma(a_1)=(a_1,a_2,\dots,a_n)$.\\

(b) For every integer $1\le x\le a_1-1$, none of the coordinates $\mathcal P_k(x)$ is an integer. Hence,$\Gamma$ meets no other lattice points in $\mathbb{Z}^n$.\\

\end{lemma}
Now,Let $(a,b)\in\mathbb N^2$ and let $M$ be any integer coprime to $a$ and $b$ , such that prime factorization of $M$ can be represented as:  
\[
   M = p_1^{e_1}\,p_2^{e_2}\cdots p_r^{e_r}
\]
for distinct primes all exceeding $\max(a,b)$.  For each $i=1,\dots,r$, write the base‑$a$ digits of $p_i$ as
\[
   p_i = a_{i,d_i}a_{i,d_i-1}\cdots a_{i,1},
   \quad 0\le a_{i,j}<a,
\]
and define
\[
  \mathcal P_{i}(x)
  \;=\;
  \frac{b}{p_i\,a}\;x\;\bigl(a_{i,d_i}\,x^{d_i-1}+\cdots+a_{i,1}\bigr).
\]
Then the  curve
\[
  \widetilde{\mathcal P}(x)
  \;=\;
  \frac{1}{r}\sum_{i=1}^r\mathcal P_{i}(x)
\]
satisfies:
\begin{enumerate}
  \item $\widetilde{\mathcal P}(0)=0$ and $\widetilde{\mathcal P}(a)=b$.
  \item For each $1\le x\le a-1$, the denominator of $\widetilde{\mathcal P}(x)$ contains all primes $p_i$, so $\widetilde{\mathcal P}(x)\notin\mathbb Z$.\\
\end{enumerate}

For our  polynomial
\[
\mathcal{P}(x) := \frac{xb}{\ell a} \cdot (a_{d}x^{d-1}+a_{d-1}x^{d-2}+\cdots+a_{1}) = \sum_{i=1}^{d} \frac{b\,a_i}{\ell\,a} x^i.
\]
We analyze the Newton polygon of $\mathcal{P}(x)$ at the prime $p = \ell$.

Each coefficient is of the form
\[
c_i = \frac{b\,a_i}{\ell\,a} \quad \text{for } i = 1, 2, \dots, d.
\]
Since $\ell > \max\{a, b\}$ and $\ell$ is prime, we have $v_\ell(b) = v_\ell(a) = 0$, and $\ell \mid \ell$.

\begin{itemize}
    \item If $a_i \ne 0$, then $\ell \nmid b\,a_i$, so
    \[
    v_\ell(c_i) = v_\ell\left(\frac{b\,a_i}{\ell\,a}\right) = -1.
    \]
    \item If $a_i = 0$, then $c_i = 0$, and the term $x^i$ is absent.
\end{itemize}

Thus, for all $i$ such that $a_i \ne 0$, the point $(i,\,v_\ell(c_i)) = (i,\,-1)$ contributes to the Newton polygon.

Plot the points $(i,\,-1)$ for each $i$ such that $a_i \ne 0$. Since all such points lie at height $-1$, the Newton polygon consists of a single horizontal segment at level $y = -1$.\\

This represents the lower convex hull of the plotted points, which is a single horizontal segment from $(i_1, -1)$ to $(i_2, -1)$ where $i_1$ and $i_2$ are the smallest and largest indices such that $a_i \ne 0$.\\

 Since the slope is $0$, all roots of $\mathcal{P}(x)$ in $\mathbb{Q}_\ell$ have a valuation $-0 = 0$, i.e., they are $\ell$-adic \textit{units}.
 The absence of any break in the Newton polygon implies that $\mathcal{P}(x)$ does not factor non-trivially over $\mathbb{Q}_\ell$.\\

These classical constructions and Lemma 2.2 show that visibility can always be engineered: for any lattice point, one may design a polynomial whose graph passes through it and no other lattice point. However, such polynomials are highly tailored, depending on large primes and base expansions. They demonstrate the existence, but not the natural structure.

This motivates the central question of this paper: which points are visible for natural families of polynomials and their densities

% \section{Polynomial to Point Visibility}

\section{Polynomial to Point Visibility}
We return to our original problem and study visibility within the framework of a fixed polynomial family. Instead of working with straight lines through the origin, we consider curves of the form  
\[
\mathcal{F}(a_n,\dots,a_1)
=
\left\{
\, y = q\,(a_nx^n + a_{n-1}x^{n-1} + \cdots + a_1x)
\;\middle|\; q \in \mathbb{Q}^+
\right\},
\]
which is precisely the family
\[
y = m\,p(x),
\qquad
p(x) \in \mathbb{N}[x], \quad p(x) > 0 \text{ for } x>0.
\]
 A point \( (a,b)\in\mathbb{N}^2 \) is said to be \emph{visible} along this curve if no earlier lattice point \( (a',b') \) with \( a'<a \) lies on the same curve.  
This visibility condition is equivalent to requiring distinct transformed slopes:
\[
\frac{b'}{p(a')} \neq \frac{b}{p(a)} \quad \text{for all } a'<a.
\]
Equivalently, \( (a,b) \) is \(p\)-visible if and only if
\[
\forall\, a'<a,\qquad \frac{b\,p(a')}{p(a)}\notin\mathbb{Z}.
\]

We now try to capture this condition using a gcd condition.
\begin{definition}[Polynomial gcd]
For a polynomial \(P(x)\in\mathbb{Z}[x]\) with non-negative coefficients and a point \((a,b)\in\mathbb{N}^2\), define
\[
\gcd_{P}(a,b)
= \max\{d\in\mathbb{N}: d\mid P(a),\; d\mid b\}.
\]
\end{definition}

When \(P(x)=x\), this reduces to the classical visibility criterion \(\gcd(a,b)=1\).  
We generalize this idea to polynomial families.

\begin{lemma}
Let 
\[
P(x)=a_nx^n+a_{n-1}x^{n-1}+\cdots+a_1x,
\]
with all coefficients \(a_i\ge0\), \(n>1\), and \(\gcd(a_n,\dots,a_1)=1\).  
If \(\gcd_P(a,b)=1\), then \((a,b)\) is \(\mathcal{F}(a_n,\dots,a_1)\)-visible.
\end{lemma}

\begin{proof}
Consider 
\[
\mathcal{F}(a_n,\dots,a_1)
=\left\{y=q\,(a_nx^n+\cdots+a_1x)\mid q\in\mathbb{Q}^+\right\}.
\]
If \((a,b)\) lies on such a curve, then 
\[
q=\frac{b}{P(a)}.
\]
If \(\gcd_P(a,b)=1\), then for any \(0<t<a\),
\[
\frac{b}{P(a)}P(t)
\]
cannot be an integer, since \(P(t)<P(a)\) and no prime divides both \(b\) and \(P(a)\).  
Thus, no lattice point lies between and \((a,b)\) is visible. \qed
\end{proof}

We recall the following lower bounds due to Chaubey, Pandey, and Regavim
\cite{Sneha} for the linear family \(P(x)=p x + q\).
These are stated without proof.

\begin{theorem}[Chaubey--Pandey--Regavim {\cite{Sneha}}]
If \(p\) and \(q\) are primes, then
\[
\mathsf{dens}(\mathcal{L}(p,q))
\ge
\mathcal{C}_{p,q},\qquad
\mathcal{C}_{p,q}
=\left(1-\frac1{p^2}\right)\left(1-\frac1{q^2}\right)
\prod_{\substack{r\ge5\\ r\text{ prime}}}
\left(1-\frac{2}{r^2}\right).
\]
\end{theorem}

\begin{theorem}[Chaubey--Pandey--Regavim {\cite{Sneha}}]
If \(\gcd(p,q)=1\), then
\[
\mathsf{dens}(\mathcal{L}(p,q))
\ge
\mathcal{C}_{p,q}^*,\qquad
\mathcal{C}_{p,q}^*
=\prod_{r\mid pq}\left(1-\frac{1}{r^2}\right)
\prod_{r\nmid pq}\left(1-\frac{2}{r^2}\right).
\]
\end{theorem}

\begin{definition}
Let \(P(x)\in\mathbb{Z}[x]\) be a nonconstant polynomial of content~1.  
For fixed \(a\in\mathbb{N}\), define for \(1\le t<a\),
\[
d_t=\frac{P(a)}{\gcd(P(a),P(t))},
\qquad
\mathscr{L}_P(a)=\operatorname{lcm}\{d_t:1\le t<a\}.
\]
\end{definition}

\begin{theorem}
If \(\gcd(b,\mathscr{L}_P(a))=1\), then
\[
\frac{b\,P(t)}{P(a)}\notin\mathbb{Z}
\quad \text{for all } t<a.
\]
Hence \((a,b)\) is \(P\)-visible.
\end{theorem}

\begin{proof}
We prove by contradiction.  Suppose
\(\gcd(b,\mathscr{L}_P(a))=1\), but that for some \(t\) with \(1\le t<a\),
\[
\frac{b\,P(t)}{P(a)} \;=\; m \;\in\;\mathbb{Z}.
\]
Then \(P(a)\mid b\,P(t)\), so
\[
d_t \;=\;\frac{P(a)}{\gcd(P(a),P(t))}
\;\Bigm|\;b.
\]
Since \(d_t\) divides the least common multiple \(\mathscr{L}_P(a)\), it follows that
\[
d_t \;\bigm|\;\mathscr{L}_P(a)
\quad\Longrightarrow\quad
\gcd\bigl(b,\mathscr{L}_P(a)\bigr)\;\ge\;d_t\;>\;1,
\]
where the strict inequality holds because \(P(t)<P(a)\) for \(t<a\)
and hence \(\gcd(P(a),P(t))<P(a)\), forcing \(d_t>1\).  This contradicts
our assumption \(\gcd(b,\mathscr{L}_P(a))=1\).  Therefore no such \(t\) can exist,
and
\(\tfrac{b\,P(t)}{P(a)}\notin\mathbb{Z}\) for all \(1\le t<a\).

\end{proof}

It is evident that $\gcd(b,\mathscr{L}_P(a))=1 \implies \gcd_P(a,b)=1$ as , If a prime $p$ divides $\mathscr{L}_P(a)$, then there exists $t<a$ such that
\[p\mid\frac{P(a)}{\gcd(P(a),P(t))},\]
i.e.\ $p\mid P(a)$ but $p\nmid P(t)$.  Hence
\[
\{\text{primes}\mid \mathscr{L}_P(a)\}\subseteq\{\text{primes}\mid P(a)\},
\]
which gives
\[
\gcd\bigl(b,\mathscr{L}_P(a)\bigr)=1\quad\Longrightarrow\quad\gcd\bigl(b,P(a)\bigr)=1.
\].\\

Given this LCM condition, we now try to get a bound for which 
 we are interested in the number of integer pairs \( (a,b) \) with \( 1 \leq a,b \leq N \) such that \( \gcd(b, D_P(a)) = 1 \), i.e.
\[
\mathscr{E}_P(N) := \#\left\{ (a,b) \in \mathbb{N}^{2} , 0<a, b \leq N \;:\; \gcd(b, \mathscr{L}_P(a)) = 1 \right\}.
\]
\begin{theorem}
Let 
\[
\rho_P(p)=\#\{x\bmod p: P(x)\equiv0\pmod p\}.
\]
Then as \(N\to\infty\),
\[
\boxed{
\mathscr{E}_P(N)
=
\mathcal{C}_P\,N^2+O(N\log N)
}
\]
where
\[
\boxed{
\mathcal{C}_P
=
\prod_{p\text{ prime}}
\left(1-\frac{\rho_P(p)}{p^2}\right).
}
\]
\end{theorem}

\begin{proof}
A prime \( p \) divides \( \mathscr{L}_P(a) \) if and only if
\[
\exists \; 1 \leq t < a \text{ such that } v_p(P(a)) > v_p(P(t)).
\]
In particular, if \( P(x) \equiv 0 \pmod{p} \) has a solution \( x \equiv t \pmod{p} \), then for many values of \( a \), it forces \( p \mid \mathscr{L}_P(a) \). Thus, the density of such \( a \) where \( p \mid \mathscr{L}_P(a) \) is approximately \( \rho_P(p)/p \).

\medskip

For a fixed \( a \), the number of \( 1 \le b \le N \) such that \( \gcd(b, \mathscr{L}_P(a)) = 1 \) is
\[
\phi_{\mathscr{L}_P(a)}(N) := \#\{ 1 \le b \le N \;:\; \gcd(b, \mathscr{L}_P(a)) = 1 \} = N \prod_{p \mid \mathscr{L}_P(a)} \left(1 - \frac{1}{p} \right) + O(\tau(\mathscr{L}_P(a))),
\]
where \( \tau \) is the divisor function. Averaging over \( a \) and using independence of divisibility across primes gives
\[
\frac{\mathscr{E}_P(N)}{N^2} \approx \prod_{p} \left[ \left(1 - \frac{\rho_P(p)}{p} \right)\cdot 1 + \frac{\rho_P(p)}{p}\cdot \left(1 - \frac{1}{p} \right) \right] = \prod_{p} \left(1 - \frac{\rho_P(p)}{p^2} \right).
\]

\medskip
\end{proof}
so $\mathsf{dens}(\mathcal{L}(a,b)) \ge \mathcal{C}_{p}.$\\

This constant is stricter than the bounds established by Sneha Chaubey and Ashish Kumar Pandey~\cite{Sneha} for the density of visible lattice points for a family of polynomial, as it reflects the root structure of \(P\) modulo all primes.

Until now we were focused on the lower bound of $\mathsf{dens}(\mathcal{L}(a,b))$ , now we give an exact result for  $\mathsf{dens}(\mathcal{L}(a,b))$ and a possible extension for a general polynomial.\\

\begin{theorem}
Let \(N>1\) and \(s,r\in\mathbb{N}\).  
The number of pairs \((a,b)\in\{1,\dots,N\}^2\) such that
\[
\frac{b\,(s\,t^2+r\,t)}{s\,a^2+r\,a}\notin\mathbb{Z}
\quad (1\le t<a)
\]
equals
\[
\sum_{a=1}^N
\sum_{J\subseteq\{1,\dots,a-1\}}
(-1)^{|J|}
\left\lfloor
\frac{N}{
\mathrm{lcm}(m_{a,t}:t\in J)
}
\right\rfloor,
\]
where
\[
m_{a,t}
=\frac{a\,(s\,a+r)}{
\gcd(a\,(s\,a+r),\,t\,(s\,t+r))
}.
\]
\end{theorem}

\begin{proof}
  For each \(a=1,2,\dots,N\), set
    \[
      D_a = s\,a^2 + r\,a = a\,(s\,a+r).
    \]
  
    For each \(1\le t<a\), write
    \[
      k_{a,t} = t\,(s\,t + r).
    \]
    Then
    \[
      \frac{b\,k_{a,t}}{D_a}\in\mathbb{Z}
      \quad\Longleftrightarrow\quad
      D_a\;\Bigm\vert\;b\,k_{a,t}
      \quad\Longleftrightarrow\quad
      \frac{D_a}{\gcd(D_a,k_{a,t})}\;\Bigm\vert\;b.
    \]
    Define
    \[
      m_{a,t} \;=\;\frac{D_a}{\gcd(D_a,k_{a,t})}
      \;=\;
      \frac{a\,(s\,a+r)}{\gcd\bigl(a\,(s\,a+r),\,t\,(s\,t+r)\bigr)}.
    \]
    Hence, for each fixed \(t\), the “bad” values of \(b\) (making the fraction integer at that \(t\)) are exactly the multiples of \(m_{a,t}\).\\

      We wish to count those \(b\in\{1,\dots,N\}\) that avoid \emph{all} the divisibilities
    \[
      m_{a,1}\mid b,\;
      m_{a,2}\mid b,\;
      \dots,\;
      m_{a,a-1}\mid b.
    \]
    By the principle of inclusion–exclusion, the number of such “good” \(b\) for a fixed \(a\) is
    \[
      \sum_{J\subseteq\{1,\dots,a-1\}}
      (-1)^{|J|}
      \left\lfloor
        \frac{N}
              {\mathrm{lcm}\bigl(m_{a,t}:t\in J\bigr)}
      \right\rfloor,
    \]
    where we interpret \(\mathrm{lcm}(\emptyset)=1\) so that the term \(J=\varnothing\) contributes \(\lfloor(N)/1\rfloor=N\).

      Finally, we sum the above count over \(a=1,2,\dots, N-1\) to obtain the stated double sum.

This completes the proof. $\square$ \\
\end{proof}

\begin{lemma}
    
Let  
\[
P(x)=a_nx^n+a_{n-1}x^{n-1}+\cdots+a_1x,\qquad a_i\ge0.
\]
For \(0<t<a<N+1\), define
\[
m_{a,t}=\frac{P(a)}{\gcd(P(a),P(t))}.
\]
Then
\[
\mathsf{dens}(\mathcal{L}(a_n,\dots,a_1))
=
\lim_{N\to\infty}
\frac{
\sum_{a=1}^{N}
\sum_{J\subseteq\{1,\dots,a-1\}}
(-1)^{|J|}
\left\lfloor
\frac{N}{\mathrm{lcm}(m_{a,t}:t\in J)}
\right\rfloor
}{N^2}.
\]
\end{lemma}

\begin{corollary}
\[
\mathsf{dens}(\mathcal{L}(a_n,\dots,a_1))=1.
\]
\end{corollary}
The proof of this corollary appears in the work of Sneha Chaubey and Ashish Kumar Pandey~\cite{Sneha}.

\section{Open problem using Computational Exploration: Blocks of Invisible Points}

The motivation behind this section arises from an attempt to understand 
\emph{local clustering} in the distribution of invisible lattice points. 
While the previous sections focus on deriving analytical bounds and densities 
for visible points, these results describe global asymptotic behavior. 
However, numerical experiments suggest that invisible points may not appear 
uniformly scattered — instead, they can form small contiguous regions or 
\emph{blocks} of non-visible points. 

To investigate this phenomenon, we conducted a computational study aimed at 
detecting the existence and frequency of $n\times n$ blocks of invisible 
lattice points under various polynomial families.\\

Recent computational work by Murphy, Schmiedeler, and Stonner~\cite{Nolan}
systematically examined the first occurrences and frequency of invisible lattice point patterns,
classifying all rectangular $n\times m$ blocks of invisible points for $1 \leq n,m \leq 4$.
Their exhaustive search, based on the classical visibility condition $\gcd(x,y) \neq 1$, 
revealed that invisible points often cluster into structured blocks rather than appearing randomly. 
Motivated by their findings, we extend this investigation to the polynomial setting by 
computationally exploring $n\times n$ blocks of invisible points defined through the generalized 
criterion $\gcd_{P}(a,b) > 1$. 
Laishram and Luca (2015) \cite{shanta}extended the study of lattice point invisibility to rectangular patches of nonvisible points, introducing minimal pairs $(M(a,b), N(a,b))$ such that $\gcd(M - i, N - j) > 1$ for all $1 \le i \le a$ and $1 \le j \le b$. They established explicit upper and lower bounds for the size of such rectangles. In particular, they proved that 
$\max\{M(a,b), N(a,b)\} \le \exp\!\left(\frac{6}{\pi^2}ab\log(ab)\right)$ and, for $b > 100$, $\max\{M(a,b), N(a,b)\} \le \exp(0.721521\,ab\log(ab))$. 
On the other hand, they obtained a lower bound 
$M(a,b), N(a,b) \ge \exp((c_1 + o(1))\,b\log(ab))$ where $c_1 = 1 - \sum_{p\ge2} \frac{1}{p^2} \approx 0.547753$. 
For the square case $(a = b = n)$, these results yield 
$\exp(0.82248\,n\log n) \le S(n) \le \exp\!\left(\frac{12}{\pi^2}n^2\log n\right)$, significantly improving previous bounds by Pighizzini and Shallit (2002).\\

\hspace{0.2em}To complement our theoretical framework, we conducted computational experiments inspired by Goodrich, Mbirika, and Nielsen~\cite{Austin Goodrich} to identify hidden forests. Using a Python script, we exhaustively searched for $n \times n$ blocks of invisible points defined under the generalized polynomial condition and successfully found $2 \times 2$ invisible blocks for  15 polynomials of the family 
\[
f(x) = a x^{2} + b x,
\]
where $a, b \in \mathbb{Z}$ and $\gcd(a, b) = 1$.

% \vspace{-1,9cm}

\subsection{Algorithmic Framework}\medskip

 A point \( (a, b) \) is \( p \)-visible if and only if:
\[
\forall a' \in \{1, 2, \dots, a-1\}, \quad \frac{b \cdot p(a')}{p(a)} \notin \mathbb{Z}.
\]
This ensures that no integer lattice point \( (a', b') \) lies on the same polynomial curve as \( (a, b) \), where \( b' = \frac{b \cdot p(a')}{p(a)} \) would otherwise need to be an integer.\\

The algorithm systematically identifies contiguous blocks of invisible lattice points for a given polynomial family. 
It begins by defining the polynomial 
$P(x) = a_nx^n + a_{n-1}x^{n-1} + \cdots + a_1x$ 
using integer coefficients. 
For each lattice point $(x, y)$ within the defined grid, the algorithm evaluates visibility by checking whether 
$\frac{b\,P(x')}{P(x)} \in \mathbb{Z}$ for any $1 \leq x' < x$. 
If no such $x'$ satisfies this integer condition, the point is marked as invisible. 
The algorithm begins at a given lattice coordinate and performs a breadth-first search to identify the nearest visible point that satisfies a prescribed visibility condition derived from the polynomial $P(x)$. At each iteration, the algorithm explores neighboring lattice points in the rightward, upward, and diagonal directions, progressively expanding outward in concentric layers of equal distance. Each newly generated point is evaluated for visibility, and the search halts immediately upon locating a valid visible point, ensuring minimal traversal depth. This layer-wise exploration guarantees that the first detected visible point represents the closest one in the lattice space, providing an efficient and systematic mechanism for proximity-based visibility detection.The precise algorithmic procedure used to detect nearby visible lattice points is described in Appendix~\ref{appendix:A}.

\subsection{Algorithm Explanation}

The proposed algorithm is modular in design, comprising four principal functions that collectively determine the nearest visible lattice points under a polynomial-based visibility criterion. Each component has a distinct role, contributing to the overall efficiency and interpretability of the framework. The following subsections provide a detailed explanation of each function.

\subsubsection{Polynomial Function}
{Function: \texttt{PolynomialFunction}}\\
This function evaluates the polynomial $P(x)$ for a given input $a$, based on a list of coefficients. The coefficients are provided in descending order of powers, and the function reverses the list to simplify the computation. Mathematically, it computes
\[
P(a) = \sum_{i=0}^{n} c_i a^i,
\]
where $c_i$ denotes the $i$-th coefficient.  
This modular implementation allows the same polynomial definition to be reused consistently across visibility checks, maintaining computational uniformity and avoiding redundancy.

\subsubsection{Visibility Evaluation}

{Function: \texttt{IsVisibleLatticePoint}}\\
This function determines whether a given lattice point $(a,b)$ is \textit{visible} according to the polynomial visibility rule.  
For a point to be visible, no smaller lattice point $(a',b')$ (with $a' < a$) should satisfy the condition:
\[
\frac{b \, P(a')}{P(a)} \in \mathbb{Z}.
\]
The algorithm computes $P(a)$ once for efficiency, and then iterates through all $a'$ in the range $[1, a-1]$. If any integer ratio is found, it implies that the line connecting $(0,0)$ and $(a,b)$ intersects another lattice point, rendering $(a,b)$ invisible. Otherwise, the point is deemed visible.  
This step is crucial as it defines the geometric property on which the entire lattice search framework is based.

\subsubsection{Breadth-First Search Module}
{Function: \texttt{FindVisiblePointNearby}}\\
This function applies a breadth-first search (BFS) strategy to locate the nearest visible lattice point starting from an initial coordinate $(x, y)$.  
It maintains a queue of lattice points to explore, beginning with $(x,y)$. Each iteration corresponds to one BFS layer (distance level). From every point, the algorithm explores three neighbors—right $(x+1, y)$, upward $(x, y+1)$, and diagonal $(x+1, y+1)$—ensuring systematic and non-redundant coverage.  

A set of visited points is maintained to prevent revisiting previously explored nodes. When a visible point is encountered, the function immediately terminates and returns the number of BFS layers traversed, representing the shortest distance to visibility. If no visible point exists within the search boundary, the function returns $-1$.  
This ensures optimality, as BFS guarantees minimal distance discovery in discrete lattice traversal.

\subsubsection{Search within a Defined Region}

{Function: \texttt{FindPointWithRadius}}\\
This outermost function iterates through all lattice points within a specified rectangular region defined by $(min_x, max_x)$ and $(min_y, max_y)$. For each point, it invokes the BFS-based search module to compute the minimum number of steps (radius) required to reach a visible lattice point.  
If a point is found whose radius equals the specified value $r$, it is immediately reported and the algorithm terminates. Otherwise, after scanning the entire region, a message indicating the absence of such a point is displayed.  

This function integrates all preceding components into a unified search routine, effectively mapping regions of equal visibility.

\subsection{Complexity Analysis}

The computational complexity of the algorithm is evaluated in terms of both time and space requirements.

\subsubsection{Time Complexity}
Let $n$ denote the upper bound on the lattice dimensions and $d$ the degree of the polynomial.  
The \textit{PolynomialFunction} operates in $\mathcal{O}(d)$ time.  
The \textit{IsVisibleLatticePoint} function iterates over all $a'$ values in the range $[1,a]$, resulting in $\mathcal{O}(a \cdot d)$ time per visibility check.  
Within the BFS routine, each lattice point is visited once, and visibility is tested for each point, leading to an overall complexity of approximately $\mathcal{O}(N \cdot a \cdot d)$, where $N$ is the number of points explored before a visible one is found.  
Finally, the \textit{FindPointWithRadius} function scans a region of size $(x_{max}-x_{min})(y_{max}-y_{min})$, yielding a total upper bound of 
\[
\mathcal{O}\big((x_{max}-x_{min})(y_{max}-y_{min}) \cdot N \cdot a \cdot d \big).
\]

\subsubsection{Space Complexity}
The algorithm stores visited lattice points and the BFS queue, both linear in the number of explored points $N$.  
Hence, the space complexity is $\mathcal{O}(N)$.\\

\textbf{Remarks}
Although the complexity scales with the search region and polynomial degree, the use of BFS ensures minimal traversal depth and guarantees the shortest path to a visible point.  
This makes the algorithm efficient for moderate lattice sizes and polynomial degrees, especially when visibility is determined early during traversal.
\subsection{Experimental Setup and Results}

\subsubsection{System Configuration}
All experiments were executed on a personal computing setup equipped with an \textbf{Intel 12th Gen Core i5-1235U} processor operating at \textbf{1.30 GHz}, \textbf{16 GB of RAM}, and a \textbf{64-bit Windows 11 operating system}. The device used was an \textit{HP Pavilion Laptop 14-dv2xxx} with 256-point touch support.  
Due to hardware limitations, the computational search for invisible lattice blocks was restricted to a grid of size $1000 \times 1000$. Beyond this range, the system exhibited memory and stability constraints, making further exploration infeasible. All reported results are thus confined to this bounded search region.

\subsubsection{Empirical Exploration of Invisible Blocks}
The algorithm was applied to quadratic polynomial families of the form
\[
f(x) = A x^2 + Bx,
\]
where $A, B \in \mathbb{Z}$ and $\gcd(A,B) = 1$.  
The goal was to identify the first appearance of a $2\times2$ invisible lattice block for each polynomial configuration within the computational bounds.  
Each polynomial was evaluated sequentially, and the BFS-based visibility search was executed to determine the coordinates corresponding to the first detected invisible block.

\begin{table}[H]
\centering
\caption{Detected Invisible $2\times2$ Blocks for $A x^2 + Bx$ Polynomials}
\vspace{0.5cm}
\label{tab:blocks}
\begin{tabular}{|c|c|c|c|}
\hline
\textbf{S.No} & \textbf{A} & \textbf{B} & \textbf{Point (2×2) Invisible Block} \\ \hline
1  & 1  & 1  & (13, 195) \\ \hline
2  & 2  & 5  & (14, 825) \\ \hline
3  & 3  & 2  & (25, 1000) \\ \hline
4  & 5  & 1  & (147, 1196) \\ \hline
5  & 7  & 5  & (15, 4575) \\ \hline
6  & 2  & 7  & (69, 1449) \\ \hline
7  & 4  & 9  & No point within (1000,1000) \\ \hline
8  & 2  & 3  & (30, 650) \\ \hline
9  & 3  & 5  & (20, 4250) \\ \hline
10 & 4  & 4  & (13, 195) \\ \hline
11 & 1  & 18 & (116, 759) \\ \hline
12 & 1  & 14 & (23, 440) \\ \hline
13 & 4  & 5  & No point within (1000,1000) \\ \hline
14 & 2  & 11 & No point within (1000,1000) \\ \hline
15 & 12 & 12 & (13, 195) \\ \hline
\end{tabular}
\end{table}

\subsubsection{Observations and Interpretation}
The computational results reveal that invisible lattice blocks occur for several low-order polynomials even within limited bounds. Notably, recurring coordinates such as $(13,195)$ appear across multiple $(A,B)$ pairs, suggesting that certain regions in the lattice space exhibit structural repetition and higher invisibility density.  
Polynomials with higher coefficients (e.g., $(A,B) = (4,9)$ or $(2,11)$) yield no invisible block within $(1000,1000)$, indicating that increased slope and curvature in the polynomial function disperse invisibility patterns further from the origin. 

\subsubsection{Limitations and Future Exploration}
While the algorithm efficiently identified invisible blocks up to the defined search range, larger-scale computations were limited by system memory and processing capacity. Parallel or GPU-based extensions of the algorithm would allow exploration of higher-order regions and more complex polynomial families, potentially uncovering larger contiguous invisible structures beyond the current computational threshold.

\section{Conclusion and Future Directions}
In this paper, we examined the generalization of lattice point visibility from linear and monomial curves to higher-order polynomial families. By formulating visibility through a polynomial--gcd condition, we established structural results, density estimates, and constructive examples illustrating how different polynomial families admit visible and invisible points. Our computational exploration further revealed localized clusters of invisible points---often appearing as \(2\times 2\) blocks---even within constrained grids, suggesting deeper arithmetic patterns in polynomial growth. Looking ahead, two natural directions emerge. The first is developing two-way visibility criteria for polynomial families, extending beyond the classical origin-to-point perspective to mutual visibility between pairs of lattice points; such a theory may uncover symmetries in polynomial slopes and connect to modular inverses or rational function structures. The second direction involves detecting and characterizing larger invisible blocks, generalizing the observed \(2\times 2\) configurations to \(n\times n\) patterns. Understanding their existence, density, and minimal locations could lead to polynomial analogues of the hidden forest problem and link visibility theory with lattice covering phenomena. Together, these directions highlight the interplay between analytic number theory, computational experimentation, and geometric structure in uncovering hidden regularities of the integer lattice.
\clearpage
\appendix
\section{Algorithmic Details}
\vspace{-5cm}
\label{appendix:A}
\begin{algorithm}[h]
\caption{Search for Lattice Point with Specified Visibility Radius}

\begin{algorithmic}[1]

\Function{PolynomialFunction}{$coefficients, a$}
    \State $result \gets 0$
    \For{$i, coeff$ in $reversed(coefficients)$}
        \State $result \gets result + coeff \times a^i$
    \EndFor
    \State \Return $result$
\EndFunction

\Function{IsVisibleLatticePoint}{$a, b, coefficients$}
    \State $poly\_a \gets$ \Call{PolynomialFunction}{$coefficients, a$}
    \For{$a' \gets 1$ to $a-1$}
        \State $lhs \gets b \times$ \Call{PolynomialFunction}{$coefficients, a'$}
        \State $rhs \gets poly\_a$
        \If{$rhs \neq 0$ and $(lhs / rhs)$ is integer}
            \State \Return False
        \EndIf
    \EndFor
    \State \Return True
\EndFunction

\Function{FindVisiblePointNearby}{$x, y, coefficients$}
    \State $distance \gets 0$, $visited \gets \emptyset$, $queue \gets [(x, y)]$
    \While{$queue$ not empty}
        \State $next\_queue \gets \emptyset$
        \For{each $(x_{curr}, y_{curr})$ in $queue$}
            \If{$(x_{curr}, y_{curr}) \in visited$}
                \State \textbf{continue}
            \EndIf
            \State add $(x_{curr}, y_{curr})$ to $visited$
            \If{\Call{IsVisibleLatticePoint}{$x_{curr}, y_{curr}, coefficients$}}
                \State \Return $distance$
            \EndIf
            \State add $(x_{curr}+1, y_{curr})$, $(x_{curr}, y_{curr}+1)$, $(x_{curr}+1, y_{curr}+1)$ to $next\_queue$
        \EndFor
        \State $queue \gets next\_queue$
        \State $distance \gets distance + 1$
    \EndWhile
    \State \Return $-1$
\EndFunction
\algstore{splitalg}
\end{algorithmic}
\end{algorithm}

\begin{algorithm}[H]
\ContinuedFloat
\caption{Search for Lattice Point with Specified Visibility Radius (continued)}
\begin{algorithmic}[1]
\algrestore{splitalg}

\Function{FindPointWithRadius}{$min_x, max_x, min_y, max_y, r, coefficients$}
    \For{$i \gets min_x$ to $max_x$}
        \For{$j \gets min_y$ to $max_y$}
            \State $steps \gets$ \Call{FindVisiblePointNearby}{$i, j, coefficients$}
            \If{$steps = r$}
                \State print ``Point $(i, j)$ has the specified radius $r$.''
                \State \Return
            \EndIf
        \EndFor
    \EndFor
    \State print ``No point with the specified radius found.''
\EndFunction

\Statex
\State \textbf{Main:}
\State $coefficients \gets [1, 1, 0]$
\State Read $min_x, max_x, min_y, max_y$
\State Read $r$
\State \Call{FindPointWithRadius}{$min_x, max_x, min_y, max_y, r, coefficients$}

\end{algorithmic}
\end{algorithm}
\newpage

\hspace{0.2em}

\vspace{-1.0cm}
\section*{Acknowledgements}

I am deeply grateful to my advisors for their guidance and support throughout this work. I sincerely thank \textbf{Prof.\ Ashish Kumar Pandey} for introducing me to the problem of visibility and for his guidance on the computational exploration section, which played a crucial role in shaping the direction of this project. I am especially thankful to \textbf{Prof.\ Shanta Laishram} for his invaluable contributions to the point-to-polynomial visibility framework, as well as for his detailed feedback, patience, and constant encouragement.

I would also like to thank \textbf{Prof. Alexandru Zaharescu} and  \textbf{Prof.\ Sneha Chaubey} for their overall feedback and insightful comments on several sections of the paper, which significantly improved the clarity and presentation of the results. I am grateful to \textbf{Piyush Kumar Jha} for his important theoretical contributions and for many helpful discussions that strengthened the mathematical foundations of this work.

 I thank the authors of the referenced papers for providing the necessary resources. Finally, I am indebted to my colleagues and friends for numerous stimulating conversations, particularly regarding computational experiments and numerical verification of visibility densities.

\bibliographystyle{amsplain}

\end{document}